\documentclass[12pt]{amsart}

\usepackage[hmargin=0.8in,height=8.6in]{geometry}
\usepackage{amssymb,amsthm, times}
\usepackage{delarray,verbatim}
\usepackage{srcltx}
\usepackage{ifpdf}
\ifpdf
\usepackage[pdftex]{graphicx}
 \else
\usepackage[dvips]{graphicx}
 \fi

\usepackage{ifpdf}
\usepackage{color}
\definecolor{webgreen}{rgb}{0,.5,0}
\definecolor{webbrown}{rgb}{.8,0,0}
\definecolor{emphcolor}{rgb}{0.95,0.95,0.95}

\usepackage{hyperref}
\hypersetup{%
          colorlinks=true,
          linkcolor=webbrown,
          filecolor=webbrown,
          citecolor=webgreen,
          breaklinks=true}
\ifpdf \hypersetup{pdftex,
            pdfstartview=FitH, %%Fit, FitB, FitH
            bookmarksopen=true,
            bookmarksnumbered=true
} \else \hypersetup{dvips} \fi

\numberwithin{equation}{section}

\newtheorem{proposition}{Proposition}[section]

\newtheorem{remark}{Remark}[section]
\newtheorem{lemma}{Lemma}[section]

\newtheorem{assump}{Assumption}[section]

\numberwithin{remark}{section} \numberwithin{proposition}{section}
\numberwithin{corollary}{section}

\renewcommand{\S}{\mathcal{S}}

\newcommand {\R}{\mathbb{R}}

\newcommand {\F}{\mathcal{F}}

\newcommand {\p}{\mathbb{P}}

\newcommand {\E}{\mathbb{E}}

\newcommand{\diff}{{\rm d}}

\newcommand{\1}{\mbox{1}\hspace{-0.25em}\mbox{l}}
\newcommand{\lev}{L\'{e}vy }

\title{Optimal Stopping When the Absorbing Boundary is Following
 After}
\author[M. Egami]{Masahiko Egami}
\address[M. Egami]{Graduate School of Economics,
Kyoto University, Sakyo-Ku, Kyoto, 606-8501, Japan}
\email{egami@econ.kyoto-u.ac.jp}
\urladdr{http://www.econ.kyoto-u.ac.jp/{\textasciitilde}egami/}
\thanks{First draft: September 3, 2012 ; this version: August 18, 2014.  \\ TThis work is in part supported
by Grant-in-Aid for Scientific Research (B) No. 23330104 and No. 26285069, Japan Society for the Promotion of Science.
We thank the participants for their valuable comments in the workshop ``Topics in \lev and jump models" at Osaka University.}
\author[T. Oryu]{Tadao Oryu}
\address[T. Oryu]{Graduate School of Economics,
Kyoto University, Sakyo-Ku, Kyoto, 606-8501, Japan}
\email{oryu.tadao.27r@st.kyoto-u.ac.jp}

\date{}
\begin{document}
\begin{abstract}
We consider a new type of optimal stopping problems where the absorbing boundary moves as the state process $X$ attains new maxima $S$.  More specifically, we set the absorbing boundary as $S - b$ where $b$ is a certain constant.  This problem is naturally connected with excursions from zero of the reflected process $S-X$.  We examine this constrained optimization with the state variable $X$ as a spectrally negative \lev process.
The problem is in nature a two-dimensional one.  The threshold strategy given by the path of \emph{just} $X$ is not in fact optimal. It turns out, however,  that we can reduce the original problem to
an infinite number of one-dimensional optimal stopping problems, and we find explicit solutions.

This work is motivated by the bank's profit maximization with the constraint that it maintain a certain level of leverage ratio.  When the bank's asset value severely deteriorates, the bank's required capital requirement shall be violated.  This situation corresponds to $X<S-b$ in our setting.  This model may well describe a real-life situation where even a large bank can fail because the absorbing boundary is keeping up with the size of the bank.
\end{abstract}

\maketitle \noindent \small{\textbf{Key words:} Optimal stopping; excursion theory;
 spectrally negative \lev processes; scale functions.\\
%\noindent JEL Classification: G32, D81, C61 \\
\noindent Mathematics Subject Classification (2010) : Primary: 60G40
Secondary: 60J75 }\\

\section{Introduction}\label{sec:introduction}%%%%%%%%%%%%%%%%%%%%%%%%%%%%%%%%%%%%%%%%%%%%%
%\mbox{{\rm 1}}\hspace{-0.25em}\mbox{{\rm l}}
The literature about optimal stopping problems and their applications is immense. In an infinite horizon problem, with one-dimensional continuous diffusions as the state variable,
a full characterization of the value function and of optimal stopping rule is known and the methodology for solution has been established.  See, for example, Dynkin \cite{dynkin}, Alvarez \cite{alvarez2}, Dayanik and Karatzas \cite{DK2003}.
For \emph{spectrally negative \lev processes}, or \lev processes with only negative jumps, a number of authors have succeeded in extending the classical results by using the scale functions.  We just name a few here : \cite{Baurdoux2008,Baurdoux2009}  for stochastic games,  \cite{Avram_et_al_2007, Kyprianou_Palmowski_2007, Loeffen_2008}  for the optimal dividend problem, \cite{alili-kyp, avram-et-al-2004} for American and Russian options,  and \cite{Egami-Yamazaki-2010-1, Kyprianou_Surya_2007} for credit risk.  However, the solution techniques presented in each paper are more or less problem-specific and no characterization of the value function is yet  known.
If the problem involves \emph{two} state variables, then even for continuous diffusions, very few things are known in the literature.

  We study a new type of optimal stopping problems.  We let  $X=(X_t, t\ge 0)$ be a spectrally negative \lev process and denote by $Y$ the reflected process
  \[
  Y_t=S_t -X_t
  \] where $S_t=\sup _{u \in [0,t]} X_u\vee s$. We then consider an optimal stopping problem for both $X$ and $S$ in which the absorbing boundary is defined by $(S_t-b, t\ge 0)$ with $b$ as a positive constant.  This means that while $X$ grows and keeps attaining new maxima, the absorbing boundary is accompanying with $S$.  Hence an excursion from $S$, if greater than $b$, would bring $X$ to ruin.  This situation is seen in the real world; for example, several large financial institutions failed in the last crisis in 2007-2008.  One of the reasons is that, while becoming large banks,  they maintain high leverage and accordingly, the banks are not so far way from the bankruptcy threshold. Instead, the bankruptcy threshold  keeps up with the size of the banks.  That is, despite the size of the bank, the risk of bankruptcy is not so much mitigated.  This paper is motivated by this phenomenon.  See Section \ref{sec:bank-example} for details.  While we take the example of banking, one can come up with other applications of this type, as long as the absorbing boundary is determined in relation to the state process' running maxima.  For instance, a gambler may have a policy that he stops betting when his wealth $X$ goes below a certain level $b$ from that day's running maxima $S$.

An excursion theory for spectrally negative \lev processes has been developed recently.  See Bertoin \cite{Bertoin_1996} as a general reference.  More specifically, an exit problem of the reflected process $Y$ was studied by Avram et al. \cite{avram-et-al-2004}, Pistorius \cite{Pistorius_2004} \cite{Pistorius_2007} and Doney \cite{Doney_2005}.

In the above cited papers on optimal stopping problems, the optimal strategy is usually obtained by so-called ``threshold strategy". That is, the player should stop and receive rewards on the first occasion when the state process enters
a stopping region.   In \lev and other jump models, the authors first find the optimal threshold level  and then prove its optimality by verifying the `quasi-variational inequalities'.  See {\O}ksendal and Sulem \cite{sulem}.  Since  the problem at hand involves two dimensions; one is $X$ and the other is $S$, finding and proving the overall optimal strategy may be challenging (as mentioned, no characterization in two-dimensional problems has been found). Recent developments on two-dimensional optimal stopping problems (involving $S$ and $X$) include Ott \cite{ott_2013} and Guo and Zervos \cite{Guo-Zervos_2010}.  In the former paper, the author solves problems including a capped version of the Shepp-Shiryaev problem \cite{shepp-shiryaev-1993}, and the latter paper is another contribution that extends \cite{shepp-shiryaev-1993}.

In our particular situation, by looking at excursions that occur at each level of $S$, we reduce the problem to an infinite number of one-dimensional optimal stopping problems.  We shall then find an explicit form of the solution, thanks to the results by Pistorius \cite{Pistorius_2007}.  It turns out that the optimal stopping region can be shown  in a diagram created by various values of $S$ and $S-X$ (e.g., Figure \ref{l(m)_(x,s)}).

The rest of the paper is organized as follows.  In Section
\ref{sec:model}, we formulate a mathematical model with a review of some important facts of spectrally negative \lev processes, and then
find an optimal threshold level in Section \ref{sec:solution}.  We shall take the example of a bank's optimization in Section \ref{sec:bank-example} and provide an explicit calculation.

\section{Mathematical Model}\label{sec:model} %%%%%%%%%%%%%%%%%%%%%%%%%%%%%%%%%%%%%%%

Let the spectrally negative Levy process $X=\{X_t;t\geq 0\}$ represent the state variable defined on the probability space
$(\Omega, \F, \p)$, where $\Omega$ is the set of all possible realization of the
stochastic economy, and $\p$ is a probability measure defined on $\F$. We denote by
$\mathbb{F}=\{\F_t\}_{t\ge 0}$ the filtration with respect to which $X$ is adapted and with the usual
conditions being satisfied. The Laplace exponent $\psi$ of $X$ is given by
\[
\psi(\lambda)=\mu\lambda+\frac{1}{2}\sigma^2\lambda^2+\int_{(-\infty,0)}(e^{\lambda x}-1-\lambda x \1_{(x>-1)})\Pi(\diff x),
\]
where $\mu \geq 0$, $\sigma \geq 0$, and $\Pi$ is a measure concentrated on $\R\backslash \{0\}$ satisfying
$\int_{\R}(1\wedge x^2)\Pi(\diff x)<\infty$. It is well-known that $\psi$ is zero at the origin, convex on $\R_+$
and has a right-continuous inverse:
\[
\Phi(q) :=\sup\{\lambda \geq 0: \psi(\lambda)=q\}, \quad q\ge 0.
\]

The running maximum process $S=\{S_t;t\geq 0\}$ is defined by
$S_t=\sup _{u \in [0,t]} X_u\vee s$. In addition, we write $Y$ for the reflected process defined by $Y_t=S_t-X_t$,
and let $\zeta$ be the stopping time defined by
\[\zeta:=\inf \{t\geq 0 : Y_t > b\}, \quad b>0,\]
the time of ruin.  The payoff is composed of three parts; the running income to be received continuously until stopped or absorbed, the terminal reward part
to be received when the process is stopped, and the penalty part incurred when the process is absorbed.

We consider the following optimal stopping problem and the value function $\bar{V}:\R^2 \mapsto \R$ associated with initial values $X_0=x$ and $S_0=s$;
\begin{eqnarray}\label{problem}
\bar{V}(x,s)&=&\sup_{\tau\in\S} \E^{x,s} \left[ \int^{\tau\wedge \zeta}_0 e^{-qt}f(X_t,S_t)\diff t\right.\\ \nonumber
&&\left.+ \1_{\{\tau<\zeta\}}e^{-q\tau}g(X_{\tau},S_{\tau})-\1_{\{\tau\geq\zeta\}}e^{-q\zeta}k(X_{\zeta},S_{\zeta})  \right]
\end{eqnarray}
where $\mathbb{P}^{x,s}(\,\cdot\,):=\mathbb{P}(\,\cdot\,|\,X_0=x, S_0=s)$ and $\E^{x,s}$ is the expectation operator corresponding to $\mathbb{P}^{x,s}$, $q\geq 0$ is the constant discount rate and $\S$ is the set of all $\mathbb{F}$-adapted stopping times.  The running income function
$f:\R^2 \mapsto \R$ is a measurable function that satisfies
\[
\E^{x,s}\left[\int_0^\infty e^{-qt}|f(X_t,S_t)|\diff t\right]<\infty.
\]
The reward function $g:\R^2 \mapsto \R_+$ and the penalty function $k:\R^2 \mapsto \R_+$ are assumed to be measurable. Our main purpose is to calculate $\bar{V}$ and to find the stopping time $\tau^*$ which attains the supremum.

For each Borel measurable function $l:\R\mapsto \R_+$, we define a stopping time $\tau(l)$ by
\begin{equation}\label{eq:tau-l}
\tau(l):=\inf\{t\geq 0:S_t-X_t>l(S_t)\},
\end{equation}
and define a set of stopping times $\S'$ by \[\S':=\{\tau(l):l:\R\mapsto \R_+\}.\]
In other words, $\tau(l)$ is the first time the excursion $S-X$ from level, say $S=s$, becomes greater than some value $l(s)$.
%We shall show later that it suffices to consider stopping times of the form (\ref{eq:tau-l}), i.e., $\tau^*\in\S'$.

When $l$ is constant, for example, $\bar{l}\equiv c$ on $\R$, we write
\[
\tau_c:=\inf\{t\ge 0: S_t-X_t>c\}.
\]
  In particular, if $\bar{l}\equiv b$, then $\tau_b=\zeta$.
Next we let $\S'(b)$ be the set of stopping times defined by
\[
\S'(b):=\{\tau(l):l(m)\leq b \;{\rm for\; all }\;m\in \R\}.
\]
Note that if $\tau\in\S'(b)$, then $\tau\leq\zeta$. The following lemma shows that it suffices to consider stopping times $\tau\in\S'(b)$.
\begin{lemma}
%If $\tau^*\in\S'$, then $\tau^*\in\S'(b)$.
Let us define $u: \R^2\times \S' \mapsto \R$ by
\begin{eqnarray*}
u(x, s; \tau)&:=&\E^{x, s}\left[ \int^{\tau\wedge \zeta}_0 e^{-qt}f(X_t,S_t)\diff t \right.\\
&&\left.+ \1_{\{\tau<\zeta\}}e^{-q\tau}g(X_{\tau},S_{\tau})-\1_{\{\tau\geq\zeta\}}e^{-q\zeta}k(X_{\zeta},S_{\zeta})  \right]
\end{eqnarray*}
Then for any $\tau\in \S'$, we can find a $\tau'\in \S'(b)$ such that $u(x, s; \tau)=u(x, s; \tau')$.
\end{lemma}
\begin{proof}
Set functions $l_1$ and $l_2$ by $l_1(s)>l_2(s)=b$ for some $s$ and $l_1(m)=l_2(m)$ on $m\neq s$. Then it is clear from the definition that $\tau(l_1)<\zeta$ if and only if $\tau(l_2)<\zeta$, and $\tau(l_1)=\tau(l_2)$ on $\{\tau(l_1)<\zeta\}$. Hence the right hand side of
(\ref{problem}) for $\tau=\tau(l_1)$ and $\tau=\tau(l_2)$ are equal to each other. Hence the lemma is proved.
\end{proof}

\subsection{Note on the Optimal Strategy}\label{subsec:optimality}  %%%%%%%%%%%%%%%%%%%%%%%%%%%%%%%%%%%%%%%%
We will reduce the original problem \eqref{problem} to an infinite number of one-dimensional optimal stopping problem and discuss the optimality of the
proposed strategy \eqref{eq:tau-l}.  Let us denote by $\bar{f}:\R^2\mapsto\R$ the $q$-potential of $f$ where
\[
\bar{f}(x,s):=\E^{x,s} \left[ \int^{\infty}_0 e^{-qt}f(X_t,S_t)\diff t \right].
\]
From the strong Markov property of $(X,S)$, we have
\begin{eqnarray}\label{eq:potential-rewrite}
&&\E^{x,s} \left[ \int^{\tau\wedge \zeta}_0 e^{-qt}f(X_t,S_t)\diff t \right]\\
&=&\E^{x,s} \left[\int^{\infty}_0 e^{-qt}f(X_t,S_t)\diff t - \int^{\infty}_{\tau\wedge \zeta} e^{-qt}f(X_t,S_t)\diff t \right]\nonumber\\
&=&\bar{f}(x,s)-\E^{x,s} \left[\E \left[\int^{\infty}_{\tau\wedge \zeta} e^{-qt}f(X_t,S_t)\diff t \Bigm| \mathcal{F}_{\tau\wedge \zeta}\right] \right]\nonumber\\
&=&\bar{f}(x,s)-\E^{x,s} \left[e^{-q(\tau\wedge \zeta)}\E^{X_{\tau\wedge \zeta},S_{\tau\wedge \zeta}}\left[\int^{\infty}_0 e^{-qt}f(X_t,S_t)\diff t \right] \right]\nonumber\\
&=&\bar{f}(x,s)-\E^{x,s} \left[e^{-q(\tau\wedge \zeta)}\bar{f}(X_{\tau\wedge \zeta},S_{\tau\wedge \zeta})\right]\nonumber\\
&=&\bar{f}(x,s)-\E^{x,s} \left[\1_{\{\tau<\zeta\}}e^{-q\tau}\bar{f}(X_{\tau},S_{\tau})+\1_{\{\zeta\le \tau\}}e^{-q\zeta}\bar{f}(X_{\zeta},S_{\zeta})\right].\nonumber
\end{eqnarray}
Hence the value function $\bar{V}$ can be written as
\begin{eqnarray}\label{eq:value-function-rewritten}
\bar{V}(x,s)&=&\bar{f}(x,s)+V(x,s),\nonumber
\end{eqnarray}
where
\begin{align}\label{eq:V-modified}
\hspace{0.5cm} V(x,s):=\sup_{\tau\in\S}\E^{x,s}\left[\1_{\{\tau<\zeta\}}e^{-q\tau}(g-\bar{f})(X_{\tau},S_{\tau})-\1_{\{\zeta\le \tau\}}e^{-q\zeta}(k+\bar{f})(X_{\zeta},S_{\zeta})\right].
\end{align}
Since $\bar{f}(x, s)$ has nothing to do with the choice of $\tau$, we concentrate on $V(x, s)$.

By the dynamic programming principle, we can write $V(x,s)$ as
\begin{align}\label{eq:dp}
V(x,s)
%&=&\sup_{\tau\in\S}\E^{x,s}\left[\1_{\{\tau<\zeta\}}e^{-q\tau}(g-\bar{f})(X_{\tau},S_{\tau})\right.\nonumber\\
%&&\left.-\1_{\{\zeta\leq\tau\}}e^{-q\zeta}(k+\bar{f})(X_{\zeta},S_{\zeta})\right]\nonumber \\
&=\sup_{\tau\in\S}\E^{x,s}\left[\1_{\{\tau<\theta\}}\1_{\{\tau<\zeta\}}e^{-q\tau}(g-\bar{f})(X_{\tau},S_{\tau})\right.\\
&\left.-\1_{\{\zeta<\theta\}}\1_{\{\zeta\leq\tau\}}e^{-q\zeta}(k+\bar{f})(X_{\zeta},S_{\zeta})+\1_{\{\theta<\tau\wedge\zeta\}}e^{-q\theta}V(X_{\theta},S_{\theta})\nonumber\right],
\end{align}
for any stopping time $\theta\in\S$. See, for example, Pham \cite{Pham-book} page 97. Now we set $\theta =T_s$ in (\ref{eq:dp}).
For each level $S=s$ from which an excursion occurs, the value $S$ does not change during the excursion.
Hence, during the first excursion interval from $S_0=s$, $\zeta=T^-_{s-b}$ and $S_t=s$ for any $t\leq T_s$, and (\ref{eq:dp}) can be written as the following one-dimensional problem for the state process $X$;
\begin{eqnarray}\label{eq:one-dim-version}
V(x,s)&=&\sup_{\tau\in\S}\E^{x,s}\left[\1_{\{\tau<T_s\}}\1_{\{\tau<T^-_{s-b}\}}e^{-q\tau}(g-\bar{f})(X_{\tau},s)\right.\\
&&-\1_{\{T^-_{s-b}<T_s\}}\1_{\{T^-_{s-b}\leq \tau\}}e^{-q\tau}(k+\bar{f})(X_{T^-_{s-b}},s)\nonumber\\
&&\left.+\1_{\{T_s<\tau\wedge T^-_{s-b}\}}e^{-qT_s}V(s,s)\right].\nonumber
\end{eqnarray}
Now we can look at \emph{only} the process $X$ and find $\tau^*\in \S$.

In relation to \eqref{eq:one-dim-version}, we consider the following one-dimensional optimal stopping problem as for $X$ and its value function $\widehat{V}:\R^2\mapsto \R$;
\begin{align}\label{eq:Vhat}
\widehat{V}(x,s)&=\sup_{\tau\in\S}\E^{x,s}\left[\1_{\{\tau<T_s\}}\1_{\{\tau<T^-_{s-b}\}}e^{-q\tau}(g-\bar{f})(X_{\tau},s)\right.\\
&\left.-\1_{\{T^-_{s-b}<T_s\}}\1_{\{T^-_{s-b}\leq \tau\}}e^{-q\tau}(k+\bar{f})(X_{T^-_{s-b}},s)+\1_{\{T_s<\tau\wedge T^-_{s-b}\}}e^{-qT_s}K\right]\nonumber,
\end{align}
where $K\geq 0$ is a constant. Then the following lemma provides a sufficient condition for a threshold strategy to be optimal
for \eqref{eq:Vhat}. See also Theorem 2.2 in {\O}ksendal and Sulem \cite{sulem} for this type of \emph{verification theorem}.  Note that $V=\widehat{V}$ holds when $K=V(s,s)$, and $V(s,s)$ can be obtained by our solution method offered in Section 3.
\begin{lemma}\label{vi}
{\rm Fix some $s\in\R$. If there exist $z^*\in(-\infty,s]$ and a function $w\in\mathcal{C}^1((-\infty,s])\cap
\mathcal{C}^2((-\infty,s)\backslash\{z^*\})$ such that
\begin{itemize}
\item[(i)] $w(s)=K,$
\item[(ii)] $(\mathcal{A}-q)w(x)=0\text{ and }w(x)>(g-\bar{f})(x,s) \text{ on } x\in(z^*,s),$
\item[(iii)] $(\mathcal{A}-q)w(x)<0\text{ and }w(x)=(g-\bar{f})(x,s) \text{ on } x\in[s-b,z^*],$
\item[(iv)] $w(x)=-(k+\bar{f})(x,s) \text{ on } x\in(-\infty,s-b),$
\end{itemize}
then $w(x)=\widehat{V}(x,s)$ for every $x\in(-\infty,s]$ and the $\mathbb{F}$-stopping time $\tau^*=\inf\{t\geq 0;
X_t<z^*\}$ gives supremum in (\ref{eq:Vhat}).}
Note that
\[
\mathcal{A}w(x):=\mu w'(x)+\frac{\sigma^2}{2}w''(x)+\int_0^\infty\Pi(\diff y)[w(x+y)-w(x)-yw'(x)\1_{\{-1<y\}}].
\]
\end{lemma}

\begin{proof}
First we prove $w(x)\geq \widehat{V}(x,s)$ for every $x\in \R_+$. The It\^{o}'s rule (see e.g. Cont and Tankov\cite{cont-tankov} page 277) gives us
\begin{eqnarray*}
&&e^{-q(t\wedge T_s\wedge T^-_{s-b})}w(X_{t\wedge T_s\wedge T^-_{s-b}})\\&=&w(X_0)-\int^{t\wedge T_s\wedge T^-_{s-b}}_0 qe^{-qu}w(X_u)\diff u \\
&&+\mu\int^{t\wedge T_s\wedge T^-_{s-b}}_0e^{-qu}w'(X_u)\diff u+\sigma\int^{t\wedge T_s\wedge T^-_{s-b}}_0e^{-qu}w'(X_u)\diff B_u \\
&&+\frac{\sigma^2}{2}\int^{t\wedge T_s\wedge T^-_{s-b}}_0e^{-qu}w''(X_u)\diff u \\
&&+\int^{t\wedge T_s\wedge T^-_{s-b}}_0\int^{\infty}_0\diff u\Pi(\diff y)e^{-qu}[w(X_u+y)-w(X_{s-})-yw'(X_{s-})\1_{\{-1<y\}}] \\
&&+\int^{t\wedge T_s\wedge T^-_{s-b}}_0\int^{\infty}_0(M(\diff u, \diff y)-\diff u\Pi(\diff y))e^{-qu}[w(X_u+y)-w(X_{u-})] \\
&=&w(X_0)-\int^{t\wedge T_s\wedge T^-_{s-b}}_0 qe^{-qu}w(X_u)\diff u \\
&&+\int^{t\wedge T_s\wedge T^-_{s-b}}_0e^{-qu}w'(X_u)\diff u+\sigma\int^{t\wedge T_s\wedge T^-_{s-b}}_0e^{-qu}w'(X_u)\diff B_u \\
&&+\int^{t\wedge T_s\wedge T^-_{s-b}}_0e^{-qu}(\mathcal{A}-q)w(X_u)\diff u \\
&&+\int^{t\wedge T_s\wedge T^-_{s-b}}_0\int^{\infty}_0(M(\diff u, \diff y)-\diff u\Pi(\diff y))e^{-qu}[w(X_u+y)-w(X_{u-})]
\end{eqnarray*}
where we denote by $M$ the Poisson random measure associated with $X$.  Since the process $\{X_{t\wedge T_s\wedge T^-_{s-b}}:t\geq0\}$ does not leave the interval $[s-b,s]$, the  integrals with respect to the Brownian motion $B$ and the compensated jump measure (i.e., the last term) are martingales. Therefore, by taking expectations,  we have
\begin{eqnarray}\label{eq:vi1}
w(x)&=&\E^{x,s}\left[e^{-q(t\wedge T_s\wedge T^-_{s-b})}w(X_{t\wedge T_s\wedge T^-_{s-b}}) \right]\\&&-\E^{x,s}\left[\int^{t\wedge T_s\wedge T^-_{s-b}}_0e^{-qu}(\mathcal{A}-q)w(X_u)\diff u\right].\nonumber
\end{eqnarray}  By the assumption
\[
(\mathcal{A}-q)w(x)\leq 0, \quad x\in(-\infty,s],
\]
we have
\begin{align}\label{eq:vi2}
w(x)\geq\E^{x,s}\left[ e^{-q(\tau\wedge T_s\wedge T^-_{s-b})}w(X_{\tau\wedge T_s\wedge T^-_{s-b}}) \right],\; x\in(-\infty,s],\tau\in\mathcal{S}.
\end{align}
%Since the function $w$ is bounded, the limit as $n\rightarrow\infty$ of both sides and the bounded convergence theorem give,
Now, for any stopping time $\tau\in\mathcal{S}$, by using (i)-(iv),
\begin{align}\label{eq:vi3}
w(x)&\ge \E^{x,s}\left[ e^{-q(\tau\wedge T_s\wedge T^-_{s-b})}w(X_{\tau\wedge T_s\wedge T^-_{s-b}}) \right]\\
&\ge \E^{x,s}\left[\1_{\{\tau<T_s\}}\1_{\{\tau<T^-_{s-b}\}}e^{-q\tau}(g-\bar{f})(X_{\tau},s)\right.\nonumber\\
&\left.-\1_{\{T^-_{s-b}<T_s\}}\1_{\{T^-_{s-b}\leq \tau\}}e^{-qT^-_{s-b}}(k+\bar{f})(X_{T^-_{s-b}},s)+\1_{\{T_s<\tau\wedge T^-_{s-b}\}}e^{-qT_s}K\right]\nonumber.
\end{align}
Taking the supremum over the set $\mathcal{S}$, we have $w(x)\geq\widehat{V}(x,s)$

On the other hand, if we substitute $\tau^*=\inf\{t\geq 0;X_t<z^*\}$ for $\tau$ in (\ref{eq:vi1})-(\ref{eq:vi3}), then all the inequalities are satisfied as
equalities thanks to the assumptions (i)-(iv). Therefore, $w(x)=\widehat{V}(x,s)$ for every $x\leq s$.
\end{proof}

To use Lemma \ref{vi}, one usually constructs a candidate  $w(\cdot)$ and proves the required inequalities (ii) and (iii). However, for the optimal stopping
problems in spectrally negative \lev models,  this procedure is nontrivial and problem-specific, depending on various data such as functions $f, g, k$ and process $X$.
It is because no general results about the optimality of threshold strategy have been proven. Accordingly, in this paper, having thus far characterized our two-dimensional problem \eqref{problem} as a set of one-dimensional optimal stopping problems \eqref{eq:one-dim-version}, we shall focus on and contribute  to obtaining, in the  general setting,  an explicit form of solution \eqref{eq:final-form} under threshold
strategies. Note that, for our problem, optimality of \eqref{eq:final-form} for \eqref{eq:one-dim-version} (and hence for \eqref{problem}) is given by verifying the conditions in Lemma \ref{vi} with $K=V(s, s)$.

Recall that, in the linear diffusion case, a full characterization of the value function and of optimal stopping
rule are  known and the methodology for solution has been established;
an optimal stopping rule  is given by the threshold strategy in a very general setup.
See Dayanik and Karatzas \cite{DK2003}; Propositions 5.7 and 5.14. See also Pham \cite{Pham-book}; Section 5.2.3.  Hence at
least, if $X$ has no jump (that is, $X$ is Brownian motion with
drift), the solution we derive is an optimal strategy for the function
\eqref{problem}.

\subsection{Scale functions} \label{subsec:scale_functions}
We review some mathematically important facts before solving the problem.  Associated with every spectrally negative \lev process, there exists a ($q$-)scale function
%
%Associated with every spectrally negative \lev process, there exists
%a \emph{(q-)scale function}
\begin{eqnarray*}
W^{(q)}: \R \mapsto \R; \quad q\ge 0,
\end{eqnarray*}
that is continuous, strictly increasing on $[0,\infty)$  and $0$ on $(-\infty,0)$. It is
uniquely determined by
\begin{eqnarray*}%\label{eq:scale}
\int_0^\infty e^{-\beta x} W^{(q)}(x) \diff x = \frac 1
{\psi(\beta)-q}, \qquad \beta > \Phi(q).
\end{eqnarray*}
%where
%\begin{eqnarray*}
%\Phi(q) = \sup \left\{ \lambda > 0: \psi(\lambda) = q \right\}, \quad q \geq 0.
%\end{eqnarray*}

Fix $a > x > 0$ and define
\begin{equation}\label{eq:T-stop}
 T_a:=\inf\{t\ge 0: X_t> a\} \quad {\rm and}\quad T_0^{-}:=\inf\{t\ge 0: X_t
 <0\}.
\end{equation}
%If $\T_a$ is the first time the process $X$goes above $a$ and
%$\tau_0$ is the first time it goes below zero,
then we have
\begin{eqnarray}\label{eq:exit-time}
\E^x \left[ e^{-q T_a} 1_{\left\{ T_a < T_0^{-}, \, T_a < \infty
\right\}}\right] = \frac {W^{(q)}(x)}  {W^{(q)}(a)}
\end{eqnarray}
\textrm{and}
\begin{eqnarray}
\E^x \left[ e^{-q  T_0^{-}} 1_{\left\{ T_a>
 T_0^{-}, \, T_0^{-} < \infty \right\}}\right] = Z^{(q)}(x) - Z^{(q)}(a) \frac {W^{(q)}(x)}
{W^{(q)}(a)},
\end{eqnarray}
where
\begin{eqnarray*}\label{eq:Z-q-function}
Z^{(q)}(x) := 1 + q \int_0^x W^{(q)}(y) \diff y, \quad x \in \R.
\end{eqnarray*}
Here we have $Z^{(q)}(x)=1$ on $(-\infty,0]$. We also have
\begin{eqnarray}
\E^x \left[ e^{-q  T_0^{-}} \right] = Z^{(q)}(x) - \frac q {\Phi(q)}
W^{(q)}(x), \quad x > 0. \label{laplace_tau_0}
\end{eqnarray}

 In particular, $W^{(q)}$ is continuously differentiable on $(0,\infty)$ if $\Pi$ does not have atoms and $W^{(q)}$ is twice-differentiable on $(0,\infty)$ if $\sigma > 0$; see, e.g., Chan et al.\cite{Chan_2009}.  Throughout this paper, we assume the former:
\begin{assump}
We assume that $\Pi$ does not have atoms.
\end{assump}

Fix $q > 0$.  The scale function increases exponentially;
\begin{eqnarray}
W^{(q)} (x) \sim \frac {e^{\Phi(q) x}} {\psi'(\Phi(q))} \quad
\textrm{as } \; x \uparrow \infty.
\label{scale_function_asymptotic}
\end{eqnarray}
There exists a (scaled) version of the scale function $ W_{\Phi(q)}
= \{ W_{\Phi(q)} (x); x \in \R \}$ that satisfies
\begin{eqnarray}
W_{\Phi(q)} (x) = e^{-\Phi(q) x} W^{(q)} (x), \quad x \in \R \label{W_scaled}
\end{eqnarray}
and
\begin{eqnarray*}
\int_0^\infty e^{-\beta x} W_{\Phi(q)} (x) \diff x &= \frac 1
{\psi(\beta+\Phi(q))-q}, \quad \beta > 0.
\end{eqnarray*}
Moreover $W_{\Phi(q)} (x)$ is increasing, and as is clear from
(\ref{scale_function_asymptotic}),
\begin{eqnarray}
W_{\Phi(q)} (x) \uparrow \frac 1 {\psi'(\Phi(q))} \quad \textrm{as }
\; x \uparrow \infty. \label{scale_function_asymptotic_version}
\end{eqnarray}

Regarding its behavior in the neighborhood of zero, it is known that
\begin{eqnarray}\label{eq:W0}
W^{(q)} (0) = \left\{ \begin{array}{ll} 0, & \textrm{unbounded
variation} \\ \frac 1 {d}, & \textrm{bounded variation}
\end{array} \right\}, \quad
\end{eqnarray}
where $d:= \mu - \int_{(-1,0)}x\Pi(\diff x)$, and
\begin{eqnarray}
W_+^{(q)'} (0) =
\left\{ \begin{array}{ll}  \frac 2 {\sigma^2}, & \sigma > 0 \\
\infty, & \sigma = 0 \; \textrm{and} \; \Pi(0,\infty) = \infty \\
\frac {q + \Pi(0,\infty)} {d^2}, & \textrm{compound Poisson}
\end{array} \right\}; \label{at_zero}
\end{eqnarray}
see Lemmas 4.3-4.4 of Kyprianou and Surya
\cite{Kyprianou_Surya_2007}.
For a comprehensive account of the scale function, see
\cite{Bertoin_1996,Bertoin_1997, Kyprianou_2006, Kyprianou_Surya_2007}. See \cite{Egami_Yamazaki_2010_2, Surya_2008} for numerical methods for computing the
scale function.

\section{Explicit Solution}\label{sec:solution}
Now we look to an explicit solution of $\bar{V}$ for $\tau\in\S'$.  Let us introduce the probability measure $\widetilde{\p}^{x,s}$ such that the
Radon-Nikodym derivative between $\widetilde{\p}^{x,s}$ and $\p^{x,s}$ is defined by
\[
\frac{\diff \widetilde{\p}^{x,s}}{\diff \p^{x,s}}\biggm|_{\mathcal{F}_t}=e^{-qt+\Phi(q)(X_t-x)}.
\]
Under $\widetilde{\p}^{x,s}$, $X$ has the Laplace exponent $\widetilde{\psi}$ defined by
\begin{eqnarray*}
\widetilde{\psi}(\lambda)&=&\psi(\lambda+\Phi(q))-\psi(\Phi(q))\\
&=&\left(\sigma^2\Phi(q)+\mu+\int_{(-\infty,0)}x(e^{\Phi(q)x}-1)\1_{\{x>-1\}}\Pi(\diff x)\right)\lambda \\
&&+\frac{1}{2}\sigma^2\lambda^2+\int_{(-\infty,0)}(e^{\lambda x}-1-\lambda x\1_{\{x>-1\}})e^{\Phi(q)x}\Pi(\diff x).
\end{eqnarray*}
Note that since $\widetilde{\psi}'(0+)=\psi'(\Phi(q)+)>0$, $X$ drifts to $\infty$ for $q\geq 0$.

Let $W_{\Phi(q)}:\R \mapsto \R$ be the scale function of $X$ under $\widetilde{\p}^{x,s}$, that is, $W_{\Phi(q)}$ has the Laplace transform
\[
\int^{\infty}_{0}e^{-\lambda x}W_{\Phi(q)}(x)dx=\frac{1}{\widetilde{\psi}(\lambda)}.
\]
%Note that due to  notational simplicity and no fear of being misunderstood, we denote
%\[\widetilde{W}(x):=W_{\Phi(q)}(x)\quad x\ge 0\]
%for the remainder of this paper.
In addition, we define the process $\eta=\{\eta_t;t\geq 0\}$ of the height of the excursion as
\[
\eta_u:= \sup\{(S-X)_{T_{u-}+w} : 0\leq w \leq T_u - T_{u-}\}, \text{ if \;}  T_u > T_{u-},
\]
and $\eta_u=0$ otherwise, where $T_{u-}:=\inf\{t\ge0:X_t\ge u\}=\lim_{m\rightarrow u-}T_{m}$.
Then $\eta$ is a Poisson point process, and we denote its
characteristic measure on $\widetilde{\p}^{x,s}$ by $\tilde{\nu}$.
It is known that there is a relation between $W_{\Phi(q)}$ and
$\tilde{\nu}$:
\begin{equation}\label{Scalefunction-Itomeasure}
W_{\Phi(q)}(x) = c \exp \left(-\int^{\infty}_{x} \tilde{\nu}
[u,\infty)\diff u \right),
\end{equation}
where $c$ is some positive constant. See Bertoin \cite{Bertoin_1996} page 195 for the explanation of this identity.

If $\tau\in\S'(b)$, equation \eqref{eq:potential-rewrite} can be written as
\begin{eqnarray*}
&&\E^{x,s} \left[ \int^{\tau\wedge \zeta}_0 e^{-qt}f(X_t,S_t)\diff t \right]\\
=&&\bar{f}(x,s)-\E^{x,s} \left[\1_{\{\tau<\zeta\}}e^{-q\tau}\bar{f}(X_{\tau},S_{\tau})+\1_{\{\tau=\zeta\}}e^{-q\tau}\bar{f}(X_{\tau},S_{\tau})\right].
\end{eqnarray*}
Accordingly,  the function $\bar{V}$ can be written as
\begin{eqnarray}\label{eq:value-function-rewritten}
\bar{V}(x,s)&=&\bar{f}(x,s)+V(x,s),\nonumber
\end{eqnarray}
where
\begin{eqnarray*}
V(x,s)=\sup_{\tau\in\S'(b)}\E^{x,s}\left[\1_{\{\tau<\zeta\}}e^{-q\tau}(g-\bar{f})(X_{\tau},S_{\tau})-\1_{\{\tau=\zeta\}}e^{-q\tau}(k+\bar{f})(X_{\tau},S_{\tau})\right].
\end{eqnarray*}
\subsection{When \bf{$X_0=S_0$}}

As a first step, we consider the case $X_0=S_0$. Set stopping times
$T_m$ as $T_m=\inf\{t\geq 0:X_t > m\}$ (Recall (\ref{eq:T-stop})).
From the strong Markov property of $(X,S)$, when $\tau(l)\in\S'(b)$ and
$S_0=X_0=s$, we have,
\begin{eqnarray}\label{eq:interim}
&&\E^{s,s}\left[\1_{\{\tau(l)<\zeta\}}e^{-q\tau(l)}(g-\bar{f})(X_{\tau(l)},S_{\tau(l)})\right]\\
&=&\int^{\infty}_s\E^{s,s}\left[\1_{\{\tau(l)<\zeta, S_{\tau(l)}\in\diff m\}}e^{-q\tau(l)}(g-\bar{f})(X_{\tau(l)},S_{\tau(l)})\right]\nonumber\\
&=&\int^{\infty}_s\E^{s,s}\left[\1_{\{T_m\leq\tau(l)\}}e^{-qT_m}\E^{m,m}\left[e^{-q\tau_{l(m)}}(g-\bar{f})(X_{\tau_{l(m)}},S_{\tau_{l(m)}})\right.\right.\nonumber\\
&&\left.\left.\times\1_{\{S_{\tau_{l(m)}}-X_{\tau_{l(m)}}\leq b, S_{\tau_{l(m)}}\in \diff m\}}\right]\right]\nonumber\\
&=&\int^{\infty}_s\E^{s,s}\left[\1_{\{S_{\tau(l)}\geq m\}}e^{-qT_m}\right]\Big((g-\bar{f})(m-l(m),m)\nonumber\\
&&\times\E^{m,m}\left[e^{-q\tau_{l(m)}}\1_{\{Y_{\tau_{l(m)}-}=l(m), S_{\tau_{l(m)}}\in \diff m\}}\right]\nonumber\\
&&+\int \!\!\! \int_A (g-\bar{f})(m-y+h,m)\nonumber\\
&&\times\left.\E^{m,m}\left[e^{-q\tau_{l(m)}}\1_{\{X_{\tau_{l(m)}}-X_{\tau_{l(m)}-}\in
\diff h, S_{\tau_{l(m)}}\in \diff m, Y_{\tau_{l(m)}-}\in \diff
y\}}\right]\right),\nonumber
\end{eqnarray}
where \[A=\{(y,h)\in\R^2 ; y-h\in[l(m),b], h<0, y\in[0,l(m)]\}.\]

Now we examine each term in the last line of (\ref{eq:interim}).
Since $X$ is a spectrally negative process and $S$ is its running
maximum process, by (\ref{Scalefunction-Itomeasure}) we have, for
$m\geq s$,
\begin{eqnarray}\label{eq:discount}
\hspace{0.5cm} \E^{s,s}\left[\1_{\{S_{\tau(l)}\geq m\}}e^{-qT_m}\right]&=&\widetilde{\E}^{s,s}\left[e^{-(m-s)\Phi(q)}\1_{\{S_{\tau(l)}\geq m\}}\right] \\
&=&e^{-(m-s)\Phi(q)}\widetilde{\mathbb{P}}^{s,s}(S_{\tau(l)}\geq m)\nonumber\\
&=&\exp\left(-\int^m_s \left(\frac{W'_{\Phi(q)}(l(u)+)}{W_{\Phi(q)}(l(u))}+\Phi(q)\right)\diff u \right)\nonumber\\
&=&\exp\left(-\int^m_s \frac{W_+^{(q)'}(l(u))}{W^{(q)}(l(u))}\diff u \nonumber
\right).
\end{eqnarray}

From  Theorems 1 and 2 in Pistorius \cite{Pistorius_2007}, we have
\begin{eqnarray*}
\E^{m,m}&\left[e^{-q\tau_{l(m)}}\1_{\{X_{\tau_{l(m)}}-X_{\tau_{l(m)}-}\in \diff h, S_{\tau_{l(m)}}\in \diff m, Y_{\tau_{l(m)}-}\in \diff y\}}\right]\\
&=\1_{\{y-h>l(m)\}}\Pi (\diff h) \left(W_+^{(q)'}(y)-\frac{W_+^{(q)'}(l(m))}{W^{(q)}(l(m))}W^{(q)}(y)\right)\diff y \diff m,
\end{eqnarray*}
and
\[\E^{m,m}\left[e^{-q\tau_{l(m)}}\1_{\{Y_{\tau_{l(m)}-}=l(m),
S_{\tau_{l(m)}}\in \diff
m\}}\right]=\frac{\sigma^2}{2}\left(\frac{W_+^{(q)'}(l(m))^2}{W^{(q)}(l(m))}-W_+^{(q)''}(l(m))\right).
\]
Putting together, if $\tau(l)\in\S'(b)$, (\ref{eq:interim}) becomes
\begin{eqnarray*}
&&\E^{s,s}\left[\1_{\{\tau(l)<\zeta, S_{\tau(l)}\in\diff m\}}e^{-q\tau(l)}(g-\bar{f})(X_{\tau(l)},S_{\tau(l)})\right]\\
&=&\exp\left(-\int^m_s \frac{W_+^{(q)'}(l(u))}{W^{(q)}(l(u))}\diff u \right)\left(\frac{\sigma^2}{2}\left(\frac{W_+^{(q)'}(l(m))^2}{W^{(q)}(l(m))}-W_+^{(q)''}(l(m))\right)\right.\\
&&\times(g-\bar{f})(m-l(m),m)+\int^{l(m)}_{0}\diff y\int^{y-l(m)}_{y-b}\Pi (\diff h) (g-\bar{f})(m-y+h,m)\\
&&\left.\times\left(W_+^{(q)'}(y)-\frac{W_+^{(q)'}(l(m))}{W^{(q)}(l(m))}W^{(q)}(y)\right)\right) \diff m.
\end{eqnarray*}
In the same way as above, if $\tau(l)\in\S'(b)$, we obtain for the second term of the expectation in (\ref{eq:value-function-rewritten})%\newpage
\begin{eqnarray*}
&&\E^{x,s}\left[\1_{\{\tau(l)=\zeta, S_{\tau(l)}\in\diff m\}}e^{-q\tau(l)}(\bar{f}+k)(X_{\tau(l)},S_{\tau(l)})\right]\\
&=&\int^{\infty}_s\E^{s,s}\left[\1_{\{T_m\leq\tau(l)\}}e^{-qT_m}\E^{m,m}\left[e^{-q\tau_{l(m)}}(\bar{f}+k)(X_{\tau_{l(m)}},S_{\tau_{l(m)}})\right.\right.\\
&&\left.\left.\times\1_{\{S_{\tau_{l(m)}}-X_{\tau_{l(m)}}> b, S_{\tau_{l(m)}}\in \diff m\}}\right]\right]\\
&=&\exp\left(-\int^m_s \frac{W_+^{(q)'}(l(u))}{W^{(q)}(l(u))}\diff u \right)\int \!\!\! \int_B(\bar{f}+k)(m-y+h,m)\\
&& \times\E^{m,m}\left[e^{-q\tau_{l(m)}}\1_{\{X_{\tau_{l(m)}}-X_{\tau_{l(m)}-}\in \diff h, S_{\tau_{l(m)}}\in \diff m, Y_{\tau_{l(m)}-}\in \diff y\}}\right]\\
&=&\exp\left(-\int^m_s \frac{W_+^{(q)'}(l(u))}{W^{(q)}(l(u))}\diff u \right)\left(\int^{l(m)}_{0}\diff y\int^{y-b}_{-\infty}\Pi (\diff h) \right.\\
&&\times(\bar{f}+k)(m-y+h,m)\left.\left(W_+^{(q)'}(y)-\frac{W_+^{(q)'}(l(m))}{W^{(q)}(l(m))}W^{(q)}(y)\right)\right) \diff m
\end{eqnarray*}
where \[B=\{(y,h)\in\R^2 ; y-h>b, h<0, y\in[0,l(m)]\}.\]

For notational simplicity, let the function $F_m(z):[0,b]\mapsto\R$ defined by
\begin{eqnarray}\label{eq:Fm}
F_m(z)&:=&\frac{\sigma^2}{2}\left(\frac{W_+^{(q)'}(z)^2}{W^{(q)}(z)}-W_+^{(q)''}(z)\right)(g-\bar{f})(m-z,m)\\
&&+\int^{z}_{0}\diff y\int^{y-z}_{y-b}\Pi (\diff h) (g-\bar{f})(m-y+h,m)\nonumber\\
&&\times\left(W_+^{(q)'}(y)-\frac{W_+^{(q)'}(z)}{W^{(q)}(z)}W^{(q)}(y)\right)\nonumber\\
&&-\int^{z}_{0}\diff y\int^{y-b}_{-\infty}\Pi (\diff h)
(\bar{f}+k)(m-y+h,m)\nonumber\\
&&\times\left(W_+^{(q)'}(y)-\frac{W_+^{(q)'}(z)}{W^{(q)}(z)}W^{(q)}(y)\right)\nonumber.
\end{eqnarray}
Hence we have, up to this point, proved the following:
\begin{proposition}
  When $X_0=S_0$, the function
$V(s,s)$ for $\tau\in \S'$ can be represented by
\begin{equation}\label{prop:1}
V(s,s)=\sup_{l}\int^{\infty}_{s}\exp\left(-\int^m_s \frac{W^{(q)}(l(u))}{W_+^{(q)'}(l(u))}\diff u \right)F_m(l(m))\diff m\\
\end{equation} where
$F_m(\cdot)$ is defined in (\ref{eq:Fm})
\end{proposition}

Recall that $l(s)$ denotes the height of the excursion $Y=S-X$
when $S=s$.  We wish to find, given $s$, the optimal height $l^*(s)$
to stop the process, and to calculate $V(s, s)$ explicitly.
\begin{proposition}\label{prop:2}
  Under $q\geq0$ and $\sigma>0$ instead of Assumption 2.1, suppose further that  $F_s: \R_+\mapsto \R$ is continuous.  Then we have
  \begin{equation}\label{eq:V-explicit}
V(s,s)=\frac{F_s(l^*(s))W^{(q)}(l^*(s))}{W_+^{(q)'}(l^*(s))},
\end{equation}
and $l^*(s)$ is the maximizer of the map $z \mapsto
\frac{F_s(z)W^{(q)}(z)}{W_+^{(q)'}(z)}$ on $[0,b]$.
\end{proposition}
\begin{remark}\normalfont
(i) Note that the maximizer $l^*(s)$ exists on $[0, b]$ since  the map $z \mapsto
F_s(z)W^{(q)}(z)/W_+^{(q)'}(z)$ is continuous due to $W^{(q)}\in C^2$ (by $\sigma>0$) and $[0,b]$ is compact.

(ii) If $q>0$, a sufficient condition for the continuity of $F_s$ is the
continuity of $f, g$ and $k$. This is a consequence of the
continuity of $x\mapsto \E^{x, s}[f(X_t, S_t)]$ for all $t\ge 0$ and $s\in \R_+$, $W^{(q)}\in
C^2$ and
\begin{equation}\label{eq:F-continuity}
|\bar{f}(x, s)-\bar{f}(y, s)|\le
q^{-1}|f(x, s)-f(y, s)|
\end{equation}
for all $x, y \in \R$. If $q=0$, sufficient conditions for the continuity of $F_s$ are more restrictive, for example, the boundedness of $f$.
\end{remark}
\begin{proof}
From the equation (\ref{prop:1}), we have for any $\epsilon >0$,
\begin{eqnarray*}
V(s,s)&=&\sup_{l}\left[\exp\left(-\int_s^{s+\epsilon} \frac{W^{(q)}(l(u))}{W_+^{(q)'}(l(u))}\diff u\right)\right.\\ &&\times\int^{\infty}_{s+\epsilon}\exp\left(-\int^m_{s+\epsilon} \frac{W^{(q)}(l(u))}{W_+^{(q)'}(l(u))}\diff u \right)F_m(l(m))\diff m\\
&&\left.+\int^{s+\epsilon}_{s}\exp\left(-\int^m_{s} \frac{W^{(q)}(l(u))}{W_+^{(q)'}(l(u))}\diff u \right)F_m(l(m))\diff m\right]\\
&=&\sup_{l}\left[\exp\left(-\int_s^{s+\epsilon} \frac{W^{(q)}(l(u))}{W_+^{(q)'}(l(u))}\diff u \right)V(s+\epsilon,s+\epsilon)\right.\\
&&\left.+\int^{s+\epsilon}_{s}\exp\left(-\int^m_{s} \frac{W^{(q)}(l(u))}{W_+^{(q)'}(l(u))}\diff u \right)F_m(l(m))\diff m\right]
\end{eqnarray*}
This expression motivates us to set $V_\epsilon:\R\mapsto\R$ as
\begin{eqnarray*}
V_\epsilon(s):=\sup_{l(s)}\left[\exp\left(-\frac{\epsilon W^{(q)}(l(s))}{W_+^{(q)'}(l(s))}\right)V(s+\epsilon,s+\epsilon)
+\epsilon F_s(l(s))\right].
\end{eqnarray*}
Then we have $\lim_{\epsilon\downarrow 0}V_\epsilon(s)=V(s,s)$. Since $\lim_{\epsilon\downarrow 0}V(s+\epsilon,s+\epsilon)=V(s,s)$,
the optimal threshold $l^*(s)$ should satisfy
\[
\lim_{\epsilon\downarrow 0}V_\epsilon(s)=\lim_{\epsilon\downarrow 0}\left[\exp\left(-\frac{\epsilon W^{(q)}(l^*(s))}{W_+^{(q)'}(l^*(s))}\right)V(s+\epsilon,s+\epsilon)
+\epsilon F_s(l^*(s))\right].
\]
From this equation, we obtain
\begin{eqnarray*}
V(s,s)&=&\lim_{\epsilon\downarrow 0}\frac{V_\epsilon(s)-\exp\left(-\frac{\epsilon W_+^{(q)'}(l^*(s))}{W^{(q)}(l^*(s))}\right)V\left(s+\epsilon, s+\epsilon\right)}{\left(1-\exp\left(-\frac{\epsilon W_+^{(q)'}(l^*(s))}{W^{(q)}(l^*(s))}\right)\right)}\\
&=&\lim_{\epsilon\downarrow 0}\frac{\epsilon F_s(l^*(s))}{\left(1-\exp\left(-\frac{\epsilon W_+^{(q)'}(l^*(s))}{W^{(q)}(l^*(s))}\right)\right)}=\frac{F_s(l^*(s))W^{(q)}(l^*(s))}{W_+^{(q)'}(l^*(s))}
\end{eqnarray*}
where the last equality is obtained by L'H\^{o}pital's rule, and hence $l^*(s)$ is the value which gives supremum to $\frac{F_s(z)W^{(q)}(z)}{W_+^{(q)'}(z)}$.
\end{proof}
\begin{remark}
{\rm $\frac{F_s(z)W^{(q)}(z)}{W_+^{(q)'}(z)}$ is the value for the strategy $l$ with $l(s)=z$ and $l=l^*$ for every $m>s$; that is, this amount is obtained when we stop if $X$ goes below $s-z$ in the excursion at level $S=s$ and use optimal strategy for the higher level $S>s$.}
\end{remark}
\subsection{When $S_0>X_0$}
Finally, let us consider the case of $S_0>X_0$. In this case, $\bar{V}$
can be represented in terms of $\bar{V}(s, s)$ as follows:
\begin{eqnarray}\label{eq:V-x-s}
\bar{V}(x,s)&=&\bar{f}(x,s)+\sup_{\tau\in\S'(b)}\E^{x,s}\left[\1_{\{T_s<\tau\}}e^{-qT_s}(\bar{V}-\bar{f})(s,s)\right.\\
&&+\1_{\{\tau<T_s\wedge\zeta\}}e^{-q\tau}(g-\bar{f})(X_{\tau},s)\nonumber\\
&&\left.-\1_{\{\zeta=\tau<T_s\}}e^{-q\tau}(k+\bar{f})(X_{\tau},s)\right]\nonumber
\end{eqnarray}
Set $\tau=\tau(l)$. Then, from (\ref{eq:exit-time}), the first term
in (\ref{eq:V-x-s}) can be written by
\[
\E^{x,s}\left[\1_{\{T_s<\tau\}}e^{-qT_s}(\bar{V}-\bar{f})(s,s)\right]=\frac{W^{(q)}(l(s)+x-s)}{W^{(q)}(l(s))}(\bar{V}-\bar{f})(s,s).
\]
For the second term, we use Theorem 1 and 2 in Pistorius\cite{Pistorius_2007} again to obtain, for $x\in [s-l(s), s]$
\begin{eqnarray*}
&&\E^{x,s}\left[\1_{\{\tau<T_s\wedge\zeta\}}e^{-q\tau}(g-\bar{f})(X_{\tau},s)\right]\\
&=&\E^{x,s}\left[e^{-q\tau_{l(s)}}\1_{\{Y_{\tau_{l(s)}-}=l(s), S_{\tau_{l(s)}}=s\}}\right](g-\bar{f})(s-l(s),s)\\
&&+\int \!\!\! \int_A(g-\bar{f})(s-y+h,s)
\E^{x,s}\left[e^{-q\tau_{l(s)}}\1_{\{X_{\tau_{l(s)}}-X_{\tau_{l(s)}-}\in \diff h, S_{\tau_{l(s)}}=s, Y_{\tau_{l(s)}-}\in \diff y\}}\right]\\
&=&\frac{\sigma^2}{2}\left(W_+^{(q)'}(l(s)+x-s)-\frac{W_+^{(q)'}(l(s))}{W^{(q)}(l(s))}W^{(q)}(l(s)+x-s)\right)\\
&&\times(g-\bar{f})(s-l(s),s)+\int^{l(s)}_{0}\diff y\int^{y-l(s)}_{y-b}\Pi (\diff h)(g-\bar{f})(s-y+h,s)\\
&&\times\left(\frac{W^{(q)}(l(s)+x-s)}{W^{(q)}(l(s))}W^{(q)}(y)-W^{(q)}(y+x-s)\right).
\end{eqnarray*}
For the third term, we have, in the same way as above,
\begin{eqnarray*}
&&\E^{x,s}\left[\1_{\{\zeta=\tau<T_s\}}e^{-q\tau}(k+\bar{f})(X_{\tau},s)\right]\\
&=&\int \!\!\! \int_B(k+\bar{f})(s-y+h,s)
\E^{x,s}\left[e^{-q\tau_{l(s)}}\1_{\{X_{\tau_{l(s)}}-X_{\tau_{l(s)}-}\in \diff h, S_{\tau_{l(s)}}=s, Y_{\tau_{l(s)}-}\in \diff y\}}\right]\\
&=&\int^{l(s)}_{0}\diff y\int^{y-b}_{-\infty}\Pi (\diff
h)(k+\bar{f})(s-y+h,s)\\
&&\times\left(\frac{W^{(q)}(l(s)+x-s)}{W^{(q)}(l(s))}W^{(q)}(y)-W^{(q)}(y+x-s)\right)
\end{eqnarray*}
for $x\in [s-l(s), s]$.

Now by combining all these terms, we can
write $\bar{V}(x, s)$ for $x\in [s-l^*(s), s]$,
\begin{subequations}\label{eq:final-form}
\begin{align}%\label{eq:final-form}
\bar{V}(x, s)&= \bar{f}(x,s)+\frac{W^{(q)}(l^*(s)+x-s)}{W^{(q)}(l^*(s))}(\bar{V}-\bar{f})(s,s)\\
&+\frac{\sigma^2}{2}(g-\bar{f})(s-l^*(s),s)\nonumber\\
&\times\left(W_+^{(q)'}(l^*(s)+x-s)-\frac{W_+^{(q)'}(l^*(s))}{W^{(q)}(l^*(s))}W^{(q)}(l^*(s)+x-s)\right)\nonumber\\
&+\int^{l^*(s)}_{0}\diff y\int^{y-l^*(s)}_{y-b}\Pi (\diff h)(g-\bar{f})(s-y+h,s)\nonumber\\
&\times\left(\frac{W^{(q)}(l^*(s)+x-s)}{W^{(q)}(l^*(s))}W^{(q)}(y)-W^{(q)}(y+x-s)\right)\nonumber\\
&-\int^{l^*(s)}_{0}\diff y\int^{y-b}_{-\infty}\Pi (\diff h)(k+\bar{f})(s-y+h,s)\nonumber\\
&\left(\frac{W^{(q)}(l^*(s)+x-s)}{W^{(q)}(l^*(s))}W^{(q)}(y)-W^{(q)}(y+x-s)\right),\nonumber
\end{align}
where $l^*(s)$ is the maximizer of the map $z \mapsto
\frac{F_s(z)W^{(q)}(z)}{W_+^{(q)'}(z)}$ on $[0,b]$. Note that $\bar{V}-\bar{f}=V$ and $V(s, s)$ is given by \eqref{eq:V-explicit} (or \eqref{prop:1}).
For the stopping regions, it becomes from \eqref{eq:V-x-s}
\begin{equation}
  \bar{V}(x, s)=g(x, s), \quad x\in [s-b, s-l^*(s)]
\end{equation}
\begin{equation}
  \bar{V}(x, s)=-k(x, s), \quad x\in (-\infty, s-b).
\end{equation}

\end{subequations}
\section{Bank's Optimization under Capital Requirements}\label{sec:bank-example} %%%%%%%%%%%%%%%%%%%%%%%%%%%%%%%%%%%%%%%%%
In this section, we solve an example. Imagine that a bank's total
asset value is represented by $e^X$. We set that the \emph{leverage
ratio}, defined as (Debt)/(Total Asset), cannot exceed $e^{-b}$. For
example, if the bank has the initial asset of $e^x=100$ with
$e^{-b}=0.8$, it has total asset of $100$ financed by debt $80$ and
equity $20$.  We can think of this ratio as the maximum leverage
ratio that is allowed by the banking regulations.  We assume that
the bank increases its asset base as long as $X=S$ where $S$ is the
running maximum of $X$ and that the bank's leverage ratio is
maintained at $0.8$.  Hence if the asset value appreciates to $120$,
then this would provide the bank with more lending opportunity since
the equity value is now $40$. With this new equity level, the bank
increases its leverage up to $0.8$, that is, total asset increasing
to $200$ financed by debt $160$ and equity $40$. Note that
$e^S=e^X=200$ and the debt level is $e^{-b}e^S=e^{S-b}=160$.    Now
if the bank's asset deteriorates due to defaults in the lending
portfolio, we would have $S-X>0$. In other words, there appears an
excursion from the level of $e^S=200$.  Since the asset level has
been pegged at $e^S=200$, the bank's equity would be wiped out when
$e^{S-b}=e^X$.  That is, when $e^X=160$ and the process is absorbed.

This model well describes a real situation where even a large bank
can fail easily as we have experienced several times, the recent and
magnified shock being the global financial crisis in 2007-2008.  After
becoming a large bank, it may still have an incentive to increase
assets, seeking for profits. The danger of becoming insolvent is
still $X=S-b$ if the bank continues to use leverage ratio of
$e^{-b}$. The absorbing boundary is always coming after.

Moreover, note that this model can incorporate the regulatory
requirements that the bank, when experiencing asset deterioration,
need to sell the assets in order to reduce the leverage.\footnote{We
are thankful to Nan Chen for pointing out this requirement.} For
example, assume that when the bank loses one dollar of asset, the
bank loses its equity by $\alpha$ and reduces its debt by
$1-\alpha$, where $\alpha\in(0,1]$. Then, at the time the equity is
wiped out, we have
\[
e^X\le e^S\left(1-\frac{1-e^{-b}}{\alpha}\right),
\]
that is, the process is absorbed when the excursion $S-X$ reaches  $-\log\left(1-\frac{1-e^{-b}}{\alpha}\right)$.
Note that an application of this setting to bank regulations is studied in Egami and Oryu \cite{Egami-Oryu2013a}.

We consider the problem (\ref{problem}) with the process \[X_t=x+\mu t+\sigma
B_t+\sum_{i=1}^{N(t)}\xi_i,\] where $B$ is a standard Brownian
motion, $N$ is a Poisson process with intensity $a$, and $\xi_i$
$(i=1,2,\ldots)$ are independent identically distributed random
variables whose distributions are exponential with parameter $\rho$
under $\p$. Note that if we put $a=0$, then $X$ is a Brownian motion with drift.
The reward functions are set by $f(x,s)=e^{x/2}$,
$g(x,s)=e^x$, and $k(x,s)=0$, so that the problem is
$V(x, s)=\sup_{\tau\in\S}\E^{x, s}\left[\int_0^{\tau\wedge\zeta} e^{-qt}(e^{X_t/2}) \diff t+ \1_{\{\tau<\zeta\}}e^{-q\tau}e^{X_\tau}\right]$.

In this case, the Laplace exponent
$\psi$ of $X$ is given by
\[
\psi(\lambda)=\mu\lambda + \frac{\sigma^2\lambda^2}{2}-\frac{a\lambda}{\rho+\lambda}.
\]
$\psi(\lambda)=q$ has three solutions $\Phi(q)$, $\alpha$, and $\beta$ (in decreasing order) and $q$-scale function $W^{(q)}$ of $X$ is represented with these values;
\[
W^{(q)}(x)=\frac{e^{\Phi(q)x}}{\psi'(\Phi(q))}+\frac{e^{\alpha x}}{\psi'(\alpha)}+\frac{e^{\beta x}}{\psi'(\beta)}.
\]
In the Brownian motion case, $\psi(\lambda)=q$ has only two solutions $\Phi(q)$ and $\alpha$, and the third term above vanishes.

For the existence of  $\bar{f}$, we need some restriction on the parameters. $\bar{f}$ can be calculated by
\begin{eqnarray*}
\bar{f}(x,s)&=&\E^{x,s}\left[\int^{\infty}_0 e^{-qt}f(X_t,S_t)\diff t\right]\\
&=&\int^{\infty}_0 e^{-qt}\E^{x,s}\left[f(X_t,S_t)\right]\diff t\\
&=&e^{\frac{x}{2}}\int^{\infty}_0 e^{(\psi(\frac{1}{2})-q)t}\diff t.
\end{eqnarray*}
Hence the condition we need is $\psi(\frac{1}{2})-q=\frac{\mu}{2}+\frac{\sigma^2}{8}-\frac{a}{2\rho+1}-q<0$, and then
\[
\bar{f}(x,s)=-\frac{e^{\frac{x}{2}}}{\frac{\mu}{2}+\frac{\sigma^2}{8}-\frac{a}{2\rho+1}-q}.
\]
%We computed the value and an optimal strategy for this problem
%with two parameter settings.

\subsection{Brownian motion} \label{sec:Bm} First, we consider a Brownian motion with drift by using parameters $\mu=0.05$, $\sigma=0.1$, $a=0$, $q=0.1$,
and $b=1$. Note that the value of $\rho$ has no effects on the problem when $a=0$.

Panels (i), (ii), and (iii) in Figure \ref{valuefn_bm} are the
graphs of $\frac{F_s(z)W^{(q)}(z)}{W_+^{(q)'}(z)}$ with $s=5,5.2141$,
and $5.3$, respectively. As shown in Proposition
\ref{prop:2}, the maximum value in each graph corresponds to
$V(s,s)$ for various values of $s$, and $z=l^*(s)$ are the
maximizers.
\begin{figure}[h]
\begin{center}
\begin{minipage}{0.31\textwidth}
\centering{\includegraphics[scale=0.32]{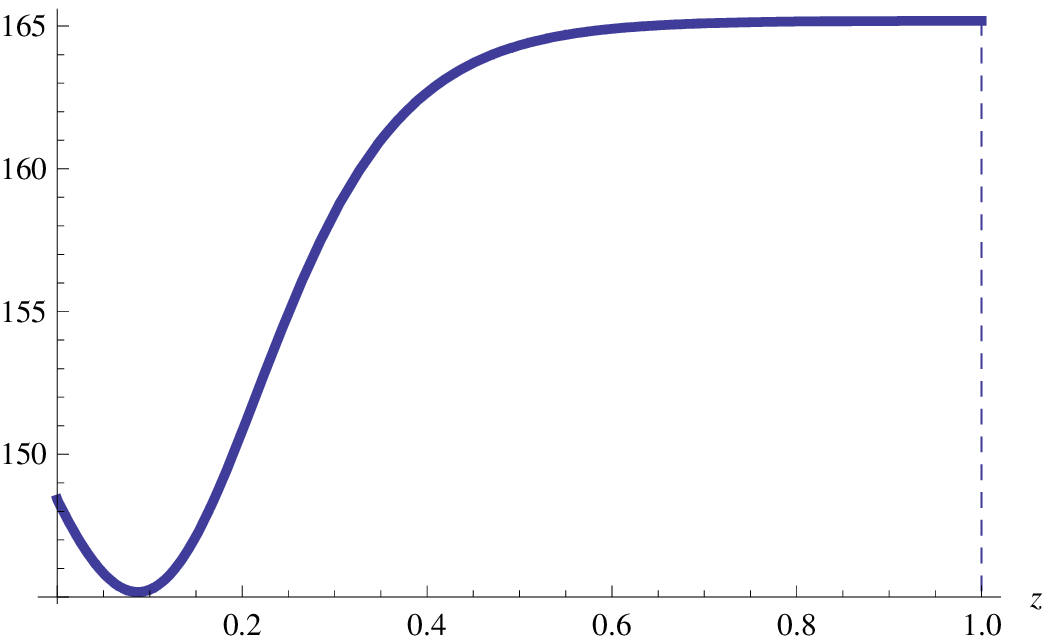}}\\
(i) s=5
\end{minipage}
\begin{minipage}{0.31\textwidth}
\centering{\includegraphics[scale=0.32]{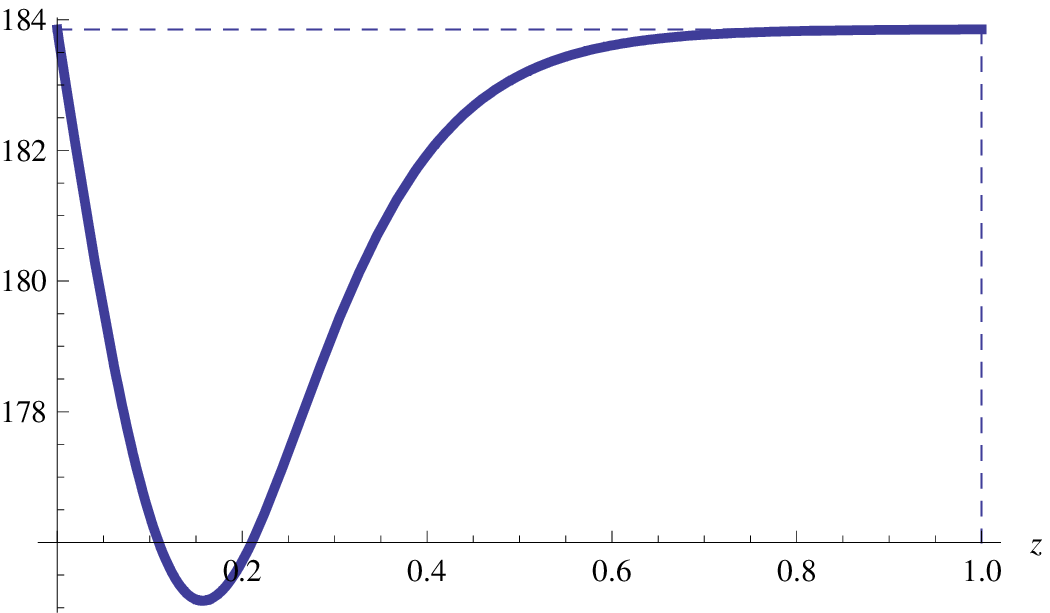}}\\
(ii) s=5.2141
\end{minipage}
\begin{minipage}{0.31\textwidth}
\centering{\includegraphics[scale=0.32]{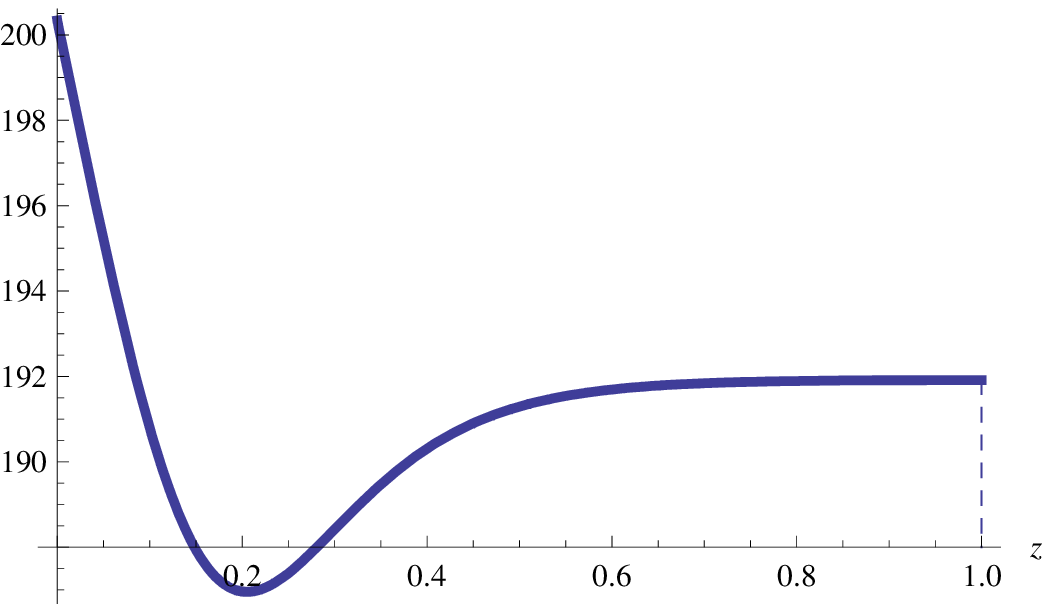}}\\
(iii) s=5.3
\end{minipage}
\caption{Graphs of $\frac{F_s(z)W^{(q)}(z)}{W_+^{(q)'}(z)}$ for various values of $s$.}
\label{valuefn_bm}
\end{center}
\end{figure}
In the case (i) $s=5$, we have  the boundary solution $l(5)=1$. This
means that, in the excursion which occurs at level $S=5$, it is
optimal to stop when $X$ goes below $5-l(5)=4$. Since we set $b=1$
here, if $X$ creeps over the level $4$, we can obtain the terminal
reward. The case (ii) $s=5.2141$ is a special point in some sense. Unlike
the case $s<5.2141$, there are two solutions $l(5.2141)=0$ and
$1$. If we choose the former strategy $l(5.2141)=0$, this
means, when $X$ reaches level $5.2141$ for the first time, we stop
it immediately and gain the terminal reward. If we choose the latter
one $l(5.2141)=1$, that means we should stop when $X$ goes below $5.2141-l(5.2141)=4.2141$.
In the case (iii), there is the boundary solution $l(5.3)=0$.
That is, when $X$ reaches level $5.3$ for the first time, we should
exit immediately and gain the terminal reward.

These arguments are summarized in Figure \ref{l(m)_bm} that
illustrates an optimal strategy $l$ over the whole region of $s\in
\R_+$. The dashed line is drawn at $s=5.2141$, which indicates the
turning points of strategies. For $s<5.2141$, $l$ constantly takes
the value of $1$. In this region, it is optimal to stop when the height
of the excursion is $1$. For $s>5.2141$, $l$ constantly takes on $0$.
In this region, our strategy reduces to the classical threshold strategy
by observing the path of process $X$. That is, stop at the first passage time of level
$5.2141$ by the process $X$.

\begin{figure}[h]
\begin{center}
\begin{minipage}{0.75\textwidth}
\centering{\includegraphics[scale=0.75]{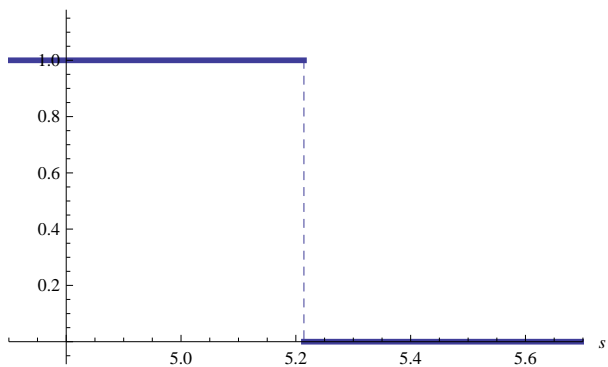}}\\
\end{minipage}
\caption{Graph of an optimal strategy $l^*$.} \label{l(m)_bm}
\end{center}
\end{figure}
\begin{figure}
\begin{center}
\begin{minipage}{\textwidth}
\centering{\includegraphics[scale=0.75]{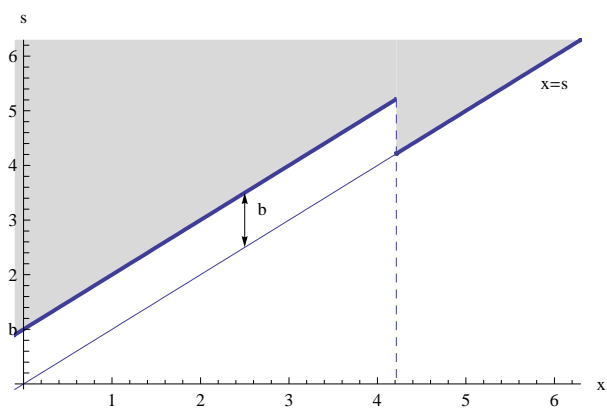}}\\
\end{minipage}
\caption{Graph of an optimal strategy $l^*$ on $(x,s)$-plane.  The shaded area is optimal stopping region.} \label{l(m)_(x,s)}
\end{center}
\end{figure}

\subsection{Brownian motion with exponential jumps}
Next, to see the effects of jumps on the threshold strategy, we set $\mu=0.25$, $\sigma=0.1$, $a=2$, $\rho=10$, $q=0.1$, and $b=1$.
Panels (i), (ii), (iii), and (iv) in Figure \ref{valuefn} are the
graphs of $\frac{F_s(z)W^{(q)}(z)}{W_+^{(q)'}(z)}$ with $s=4, 5,
5.1963$ and $5.3$, respectively. Note that the number $d$ in
(\ref{eq:W0}) is 0.05 which is equal to the drift of the Brownian motion in Section \ref{sec:Bm}.

\begin{figure}[h]
\begin{center}
\begin{minipage}{0.45\textwidth}
\centering{\includegraphics[scale=0.45]{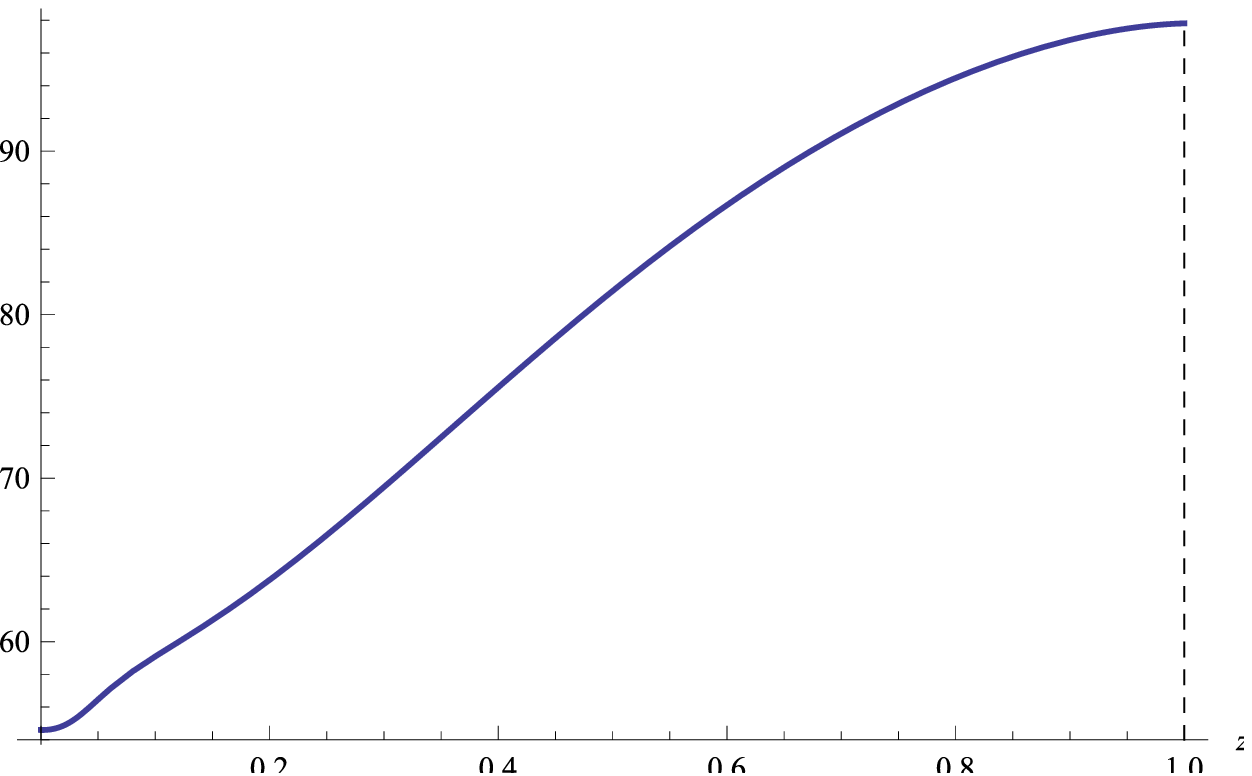}}\\
(i) s=4
\end{minipage}
\begin{minipage}{0.45\textwidth}
\centering{\includegraphics[scale=0.45]{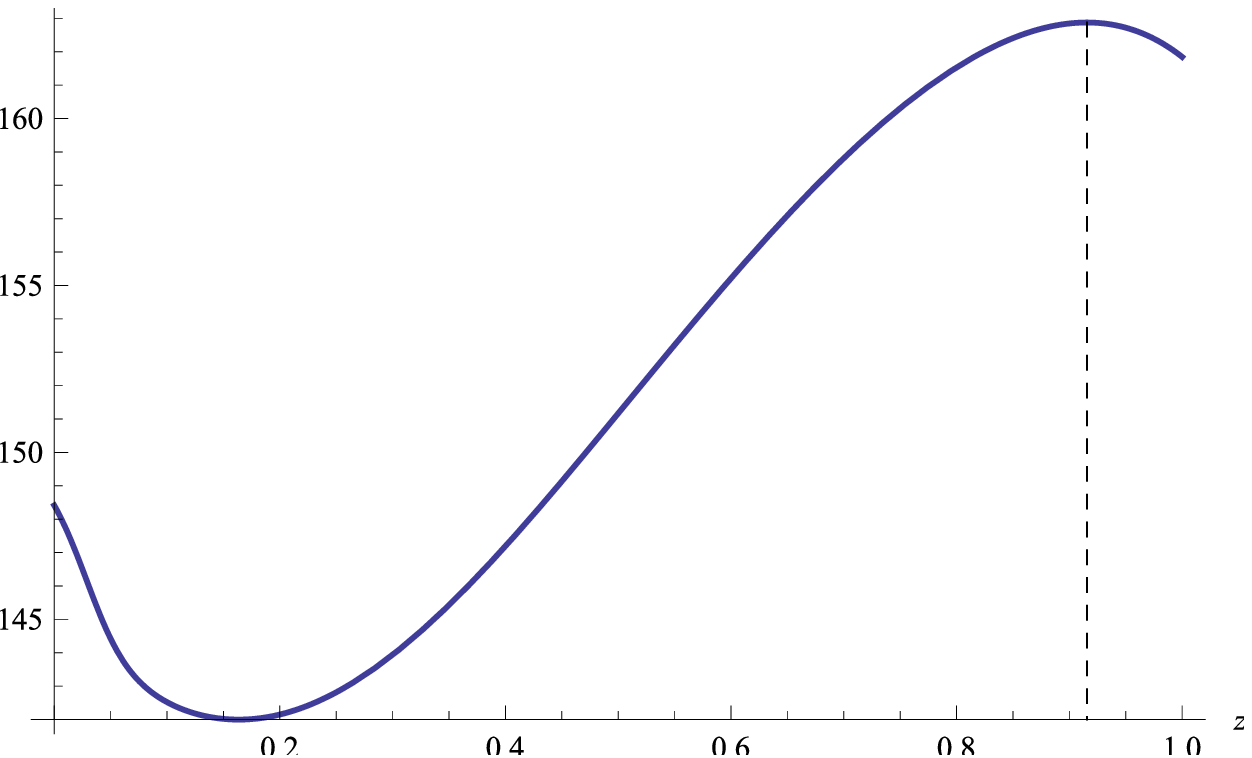}}\\
(ii) s=5
\end{minipage}
\begin{minipage}{0.45\textwidth}
\centering{\includegraphics[scale=0.45]{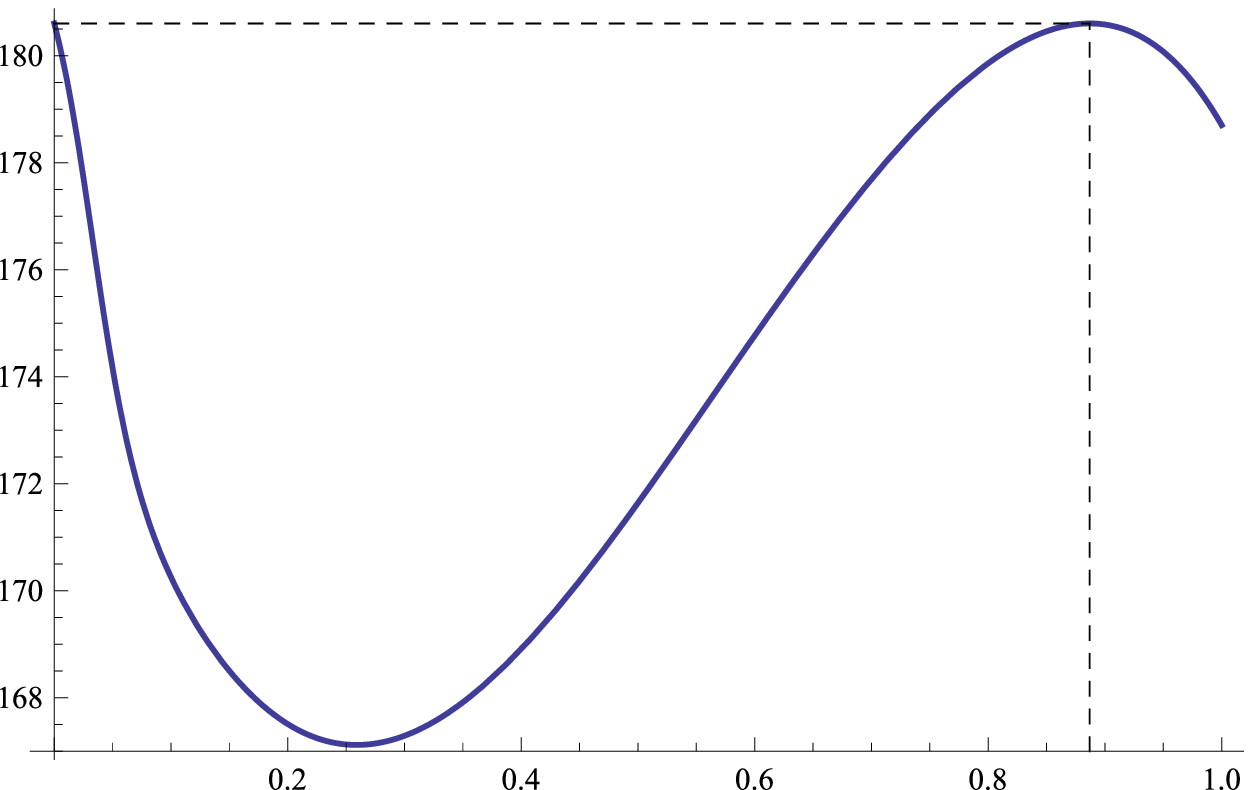}}\\
(iii) s=5.1963
\end{minipage}
\begin{minipage}{0.45\textwidth}
\centering{\includegraphics[scale=0.45]{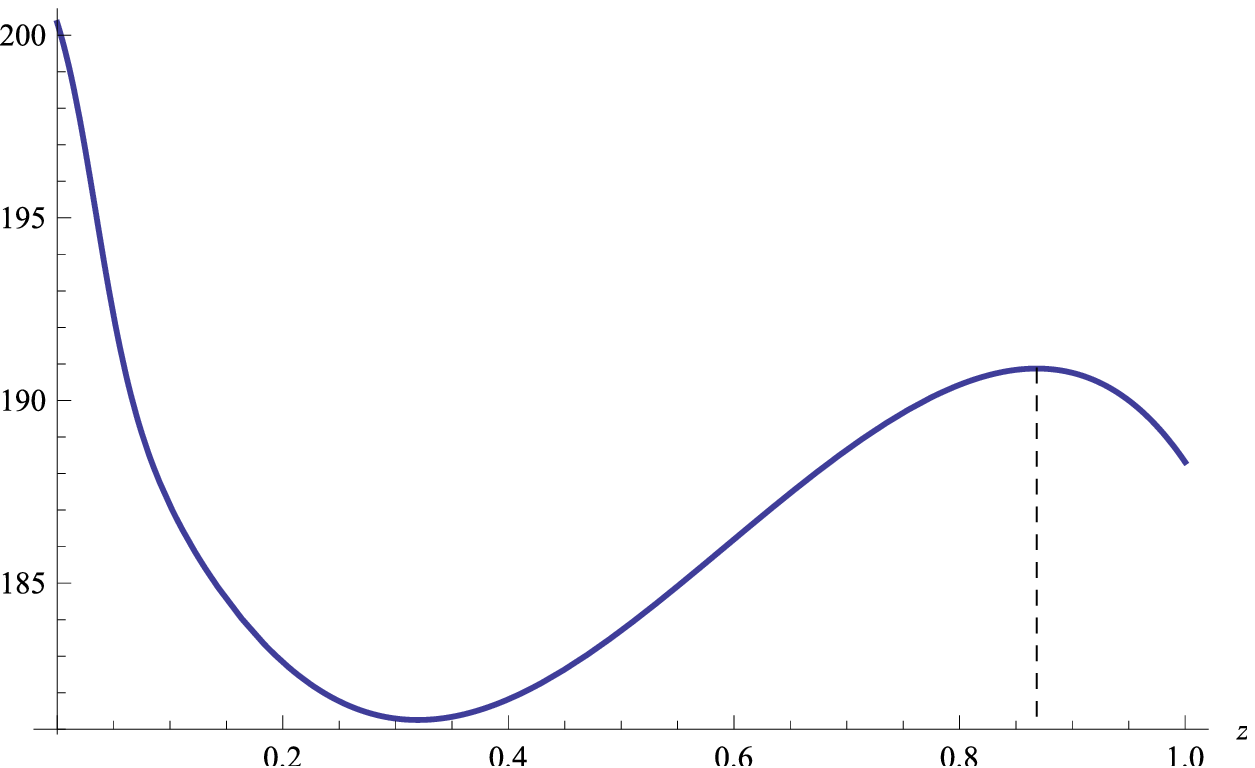}}\\
(iv) s=5.3
\end{minipage}
\caption{Graphs of $\frac{F_s(z)W^{(q)}(z)}{W_+^{(q)'}(z)}$ for various values of $s$.}
\label{valuefn}
\end{center}
\end{figure}

In the case (i) $s=4$, we have the boundary solution $l(4)=1$. This
means that, in the excursion which occurs at level $S=4$, it is
optimal to stop when $X$ goes below $4-l(4)=3$. Since we set $b=1$
here, if $X$ creeps over the level $3$, we can obtain the terminal
reward. Instead, if $X$ jumps over the level $3$, we
should pay the penalty (but we set this as $0$ here) and cannot gain
the terminal reward. In the case (ii) $s=5$, there is the internal
solution $l(5)=0.915551<1=b$. Therefore, in the excursion which
occurs at level $S=5$, it is optimal to stop immediately that $X$
goes below $5-l(5)=4.08445$. Since $l(5)<b$, if $X$ creeps over the
level $5-l(5)$ or jumps onto in the area between $5-l(5)$ and $5-b=4$, we
can obtain the terminal reward. But if $X$ jumps across
the level $4$, we cannot obtain the terminal reward.
\begin{figure}[h]
\begin{center}
\begin{minipage}{0.75\textwidth}
\centering{\includegraphics[scale=0.75]{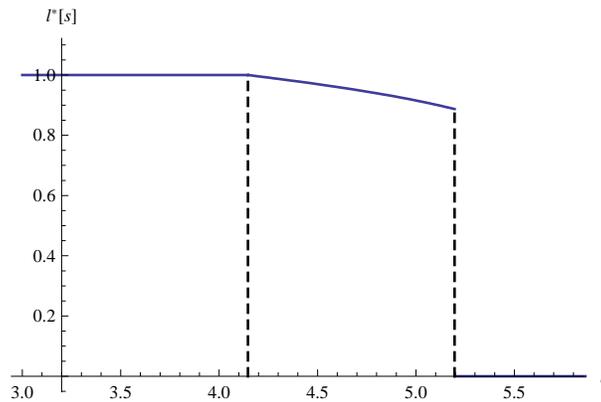}}\\
\end{minipage}
\caption{Graph of an optimal strategy $l^*$.} \label{l(m)}
\end{center}
\end{figure}
\begin{figure}
\begin{center}
\begin{minipage}{\textwidth}
\centering{\includegraphics[scale=0.75]{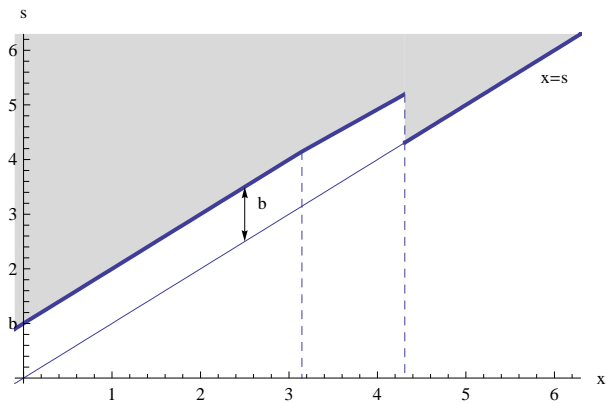}}\\
\end{minipage}
\caption{Graph of an optimal strategy $l^*$ on $(x,s)$-plane.  The shaded area is optimal stopping region.} \label{l(m)_(x,s)}
\end{center}
\end{figure}
The case (iii) $s=5.1963$ is a turning point of strategy like the case (ii)
in Figure \ref{valuefn_bm}. There are two solutions $l(5.1963)=0$ and
$0.886898$. If we choose the former strategy $l(5.1963)=0$, this
means, when $X$ reaches level $5.1963$ for the first time, we stop
it immediately and gain the terminal reward. If we choose the latter
one $l(5.1963)=0.886898$, that means we should behave like the case
(ii). In the case (iv), there is the boundary solution $l(5.3)=0$.
That is, when $X$ reaches level $5.3$ for the first time, we should
exit immediately and gain the terminal reward.

These arguments are summarized in Figure \ref{l(m)} that
illustrates an optimal strategy $l$ over the whole region of $s\in
\R_+$. Two dashed lines are drawn at $s=4.1464$ and $s=5.1963$,
which indicate the turning points of the strategies. For $s<4.1464$,
$l$ constantly takes the value of $1$. In this region, it is optimal
to stop when the height of the excursion is $1$.  When $s$ lies
between $4.1464$ and $5.1963$, $l$ has the form of concave curve
started at $1$. In this region, one should stop once the height of
excursion is greater than or equal to $l(s)<1$. Note that this region does \emph{not} show up in the Brownian motion case.
Finally,  for
$s>5.1963$, $l$ constantly takes on $0$. In this region, our
strategy reduces to the classical threshold strategy by observing
the path of process $X$.  That is, stop at the first passage time of level
$5.1963$ by the process $X$.

Before concluding this section, let us make some comparisons of the two cases. In the jump case, we have two levels $s^1:=4.1464$ and
$s^2:=5.1963$ that mark the change points of strategy, while we have only one point $s^0:=5.2141$ in the no jump case.  At the level of $s\in (s^1, s^2)$, the bank should stop the process during the excursion, a strategy that does not exist in the no jump case.  It is due to the existence of jumps which could bring the process suddenly to the ruin.  We have $s_2<s_0$, which again shows that jumps make the bank more cautious; with $(a, 1/\rho)=(2, 0.1)$ it should not expect to reach a level as high as $s_0$ even with the large $\mu=0.25$.

%The one border $s^1=4.1464$ is lower
%than that of the Brownian motion case $s^0:=4.2659$, which means that the bank is more prudent in
%taking the strategy $l(s)=1$ when there are jumps. This is because this strategy is risky in the jump case; even a small jump could bring the bank to ruin when $X$ is close enough to $s-1$.
%Accordingly, the bank may want to switch to a safer strategy $l(s)\in (0, 1)$ earlier. On the other hand, the other level $s^2=5.1963$ is
%higher than $s^0=4.2659$.
%A possible explanation is that, with strategy $l(s)\in (0, 1)$, the process does not exceed the line $s-1$ by a small jump, so that the bank could be aggressive (to reach a higher level of $s^2=5.1963$), relying on
%the large $\mu=0.25$.

% AOS,AOAS: If there are supplements please fill:
%\begin{supplement}[id=suppA]
%  \sname{Supplement A}
%  \stitle{Title}
%  \slink[doi]{10.1214/00-AOASXXXXSUPP}
%  \sdatatype{.pdf}"
%  \sdescription{Some text}
%\end{supplement}

\end{document}